\documentclass{article}

\usepackage{amsmath, amssymb,amsthm, amstext}

\newtheoremstyle{break}
   {9pt}
   {9pt}
   {\itshape}
   {}
   {\bfseries}
   {.}
   {0.5em}
   {}

\newtheoremstyle{definitionbreak}
   {9pt}
   {9pt}
   {\rm}
   {}
   {\bfseries}
   {.}
   {0.5em}
   {}

\theoremstyle{break}
\newtheorem{thm}{Proposition}[section]
\newtheorem{them}[thm]{Theorem}

\newtheorem{cor}[thm]{Corollary}

\theoremstyle{definitionbreak}
\newtheorem{definition}[thm]{Definition}

\newtheorem{r&n}[thm]{Remark and Notation}

\DeclareMathOperator{\length}{length}
\DeclareMathOperator{\LC}{LC}

\DeclareMathOperator{\Proj}{Proj}
\DeclareMathOperator{\Spec}{Spec}
\DeclareMathOperator{\Sing}{Sing}

\newcommand\idq{{\mathfrak{q}}}
\newcommand\idp{{\mathfrak{p}}}

\newcommand\Epi{\twoheadrightarrow}
\newcommand\inkl{\hookrightarrow}
\newcommand\bs{\backslash}
\newcommand\tm{\subseteq}

\newcommand\ntm{\not\subseteq}
\newcommand\Nat{\mathbb{N}_0}

\newcommand\NN{\mathbb{N}}
\newcommand\PP{\mathbb{P}}

\newcommand\defin{\mathrel{\mathop:}=}

\newcommand\sat{^{\mbox{{\footnotesize\rm sat}}}}
\newcommand\sO{{\cal O}}

\newcommand\sK{\mathfrak{K}}

\newcommand{\und}{\mathop{\wedge}\mspace{2mu}}

\newcommand{\wt}[1]{\widetilde{#1}}
\newcommand{\res}{\!\!\upharpoonright}
\newcommand{\ol}[1]{\overline{#1}}

\begin{document}

\title{Equations describing the ramification of outer simple linear projections\footnote{Mathematics Subject Classification 14N05, 14Q99; Key words Simple linear projection, Ramification}}
\author{Simon Kurmann}
\date{}

\maketitle

\begin{abstract}
 We explain how to determine equations describing the ramification of an outer simple linear projection of a projective scheme in a way suited for explicite computations.
\end{abstract}

\section{Introduction}

When we try to understand a morphism $\pi: \wt Z \to Z$ of Noetherian schemes, we might turn our attention to its ramification, that is to the set \[\Sing(\pi) \defin \{z \in Z \mid \# \pi^{-1}(z) > 1\}\] where $\#\pi^{-1}(z) = \length(\sO_{\wt Z}\otimes k(z))$ is the length of the fiber $Z\times \Spec(k(z))$ over $z \in Z$. Often, the ramification of a morphism tells us much about the morphism itself, but it can be rather difficult to determine. In this article, we develope a method to describe the ramification of $\pi$ in the case that $\pi$ is an outer simple liner projection of a projective scheme $\wt Z$. By the use of Gr\"obner bases, this method is suited for explicite computations. In the case of an outer simple linear projection $\pi$ of a projective scheme $\wt Z \tm \PP^n$, the ramification of $\pi$ also can bear crucial information about the image $Z = \pi(\wt Z) \tm \PP^{n-1}$. For example, an important subclass of varieties of almost minimal degree are outer simple linear projections of rational normal scrolls, and these varieties in turn can be classified according to the ramification of the projection (see \cite{BrP}).\\
In general, we may not expect to understand easily a morphism of projective varieties. But an outer simple linear projection $\pi: \wt Z \to Z$ from $p \in \PP^n_K\bs\wt Z$ over an algebraically closed field $K$ has a crucial property that makes it accessible: Consider $\PP^{n-1}_K$ as a subspace of $\PP^n_K$ avoiding $p$. The fiber $\pi^{-1}(q) = \wt Z \cap \langle q,p\rangle$ over a closed point $q \in Z$ is a closed subscheme of finite length of the projective line $\langle q,p\rangle = \PP^1_K$ spanned by $q$ and $p$; so, it is an effective divisor on $\PP^1_K$, and we can write $\pi^{-1}(q) = \sum_{i=1}^e \lambda_i\wt q_i$ for some distinct closed points $\wt q_1, \ldots, \wt q_e \in \PP^1_K$ and integers $\lambda_1, \ldots, \lambda_e \in \NN$. Being a subscheme of $\PP^1_K$ of finite length, $\pi^{-1}(q)$ is determined by one equation. Moreover, the numbers $\lambda_1, \ldots, \lambda_e$ already determine the isomorphy class of the scheme $\pi^{-1}(q)$ (but, of course, not the divisor itself). So, if we know how those numbers vary over $Z$, we understand the ramification of $\pi$. We can also change our point of view and consider the classical question how the image $Z$ depends on the center of projection. As an example, assume that $\wt Z$ is a smooth curve and $\pi$ is birational. Then the singular locus of $Z$ consists of the points whose fibers are not of type $e =1$ and $\lambda_1 = 1$. A closed point of $Z$ with $e=2$ and $\lambda_1=\lambda_2=1$ is a node, while the type $e = 1$ and $\lambda_1 = 2$ indicates a cusp.\\
Throughout, we are using the following setup: Let $\wt Z\tm \PP^n_K$ be a projective scheme where $K$ is an algebraically closed field of characterisic 0 and $n \in \NN$. We want to study $Z = \wt\pi(\wt Z)$ in relation to the center of projection $p \in \PP^n_K\bs\wt Z$ of the simple linear projection $\wt\pi: \PP_K^n\bs\{p\}\to \PP^{n-1}_K$. We consider $\PP^{n-1}_K$ embedded into $\PP^n_K$; this embedding is not unique, but any embedding with $p \notin \PP^{n-1}_K$ yields the same situation. We denote $\pi \defin \wt\pi\res_{\wt Z}: \wt Z \Epi Z$. We now want to give a method to find equations describing the ramification of $\pi$ depending only on equations defining $\wt Z$ and $p$. We do so in two ways: First, for $k \in \NN$, define \[Z_k^\circ \,\defin\, \{q \in Z \text{ closed point}\mid \length_{\sO_{\PP^1_K}}(\sO_{\pi^{-1}(q)}) = k\},\] the locally closed set of closed points of $Z$ such that the divisor $\pi^{-1}(q)$ has degree $\lambda_1+\cdots+\lambda_e = k$. In Proposition \ref{covering} and Corollary \ref{GBcovering}, we describe a finite affine covering $(D_g)_g$ of $Z_k^\circ$, where $D_g$ is the open subset of $Z_k^\circ$ defined by the section $g$. For any closed poinst $q \in D_g$, the ramification type of the fiber $\pi^{-1}(q)$ is determined by the distribution of the linear factors of the restrictions of $g$ to $\langle q, p\rangle$. Secondly, denote \[Z_k \,\defin\, \{q \in Z \text{ closed point}\mid \length_{\sO_{\PP^1_K}}(\sO_{\pi^{-1}(q)}) \geq k\} \,=\, \bigcup_{l\geq k} Z_l^\circ.\] A partition $\lambda = (\lambda_1, \ldots, \lambda_e)$ of $k \in \NN$ is a family of integers with $\lambda_1+\cdots+\lambda_e = k$. We use the above covering to give equations defining the $\lambda$-ramification locus of $\pi$ for any partition $\lambda$:

\begin{definition}
 Let $k \in \NN$, and let $\lambda = (\lambda_1, \ldots, \lambda_e)$ be a partition of $k$. The {\it proper $\lambda$-ramification locus of $\pi$} is the locally closed set \[Z_\lambda^\circ \defin \left\{q \in Z \text{ closed point}\,\left|\,\begin{array}{c} \pi^{-1}(q) = \sum_{i=1}^e\lambda_i\wt q_i\\ \text{ with distinct closed points } \wt q_1, \ldots, \wt q_e \in \langle q,p\rangle\end{array}\right.\right\},\] 
 and the {\it $\lambda$-ramification locus of $\pi$} is the closed set \[Z_\lambda \,\defin\, \left\{q \in Z \text{ closed point}\mid \exists \wt q_1, \ldots, \wt q_e \in \langle q,p\rangle: \pi^{-1}(q) = \sum_{i=1}^e\lambda_i\wt q_i\right\} \cup Z_{k+1}.\]
\end{definition}

Note that the closed points $\wt q_1, \ldots, \wt q_e \in \pi^{-1}(q)$ in the definition of $Z_\lambda$ need not be distinct. If two of them are equal for a closed point $q \in Z_\lambda\cap Z_k^\circ$, then $q$ is also contained in $Z_\mu$ for a partition $\mu$ of $k$ such that  $\lambda$ is a refinement of $\mu$; in this situation, we call $\mu$ a {\it coarsening} of $\lambda$. Geometrically, the set $\bigcup_\mu Z_\mu^\circ$ where $\mu\neq \lambda$ runs over all coarsenings of $\lambda$ can be considered as the degeneracy locus of $Z_\lambda\cap Z_k^\circ$. Also, $Z_\lambda$ must contain $Z_{k+1}$ to be closed. It holds \[Z_\lambda^\circ = Z_\lambda\bs\left(Z_{k+1} \cup \bigcup\{Z_\mu\mid \mu \text{ a coarsening of } \lambda \text{ with } \mu \neq \lambda\right).\] If we can determine equations defining the closed sets $Z_k$ and $Z_\lambda$ for all $k \in \NN$ and all partitions $\lambda$ of $k$, then we can give a description of $Z_\lambda^\circ$ via equations of $Z_\lambda$ and $Z_\lambda\bs Z_\lambda^\circ$. The equations defining $Z_k$ can be computed using partial elimination ideals (compare \cite{G} and \cite{ego}), and we will use this to show how to compute equations defining $Z_\lambda$ in Theorem \ref{Satz} and Corollary \ref{last}.\\
Both descriptions of the ramification of $\pi$ will be given in a way suited for explicite computations using a Gr\"obner basis of the homogeneous ideal of $\wt Z$.\\
The results in this article are part of my doctoral thesis \cite{Diss}, where a more detailed account can be found.

\section{Covering $Z_k^\circ$}

In order to formulate our results, we must first introduce some notations: Let $R = K[x_0, \ldots, x_n]$ be the homogeneous coordinate ring of $\PP^n_K$. We denote the homogeneous ideal in $R$ of a closed point in $\PP^n_K$ by the corresponding gothic letter, e.g., $\idp$ and $\wt\idq_i$ are the ideals of the closed points $p$ and $\wt q_i$, respectively. Let $S \defin K[\idp_1]$ be the $K$-algebra generated by the space of linear forms $\idp_1$ of the homogeneous ideal of the closed point $p \in \PP^n_K$. Then, the simple linear projection $\wt\pi: \PP^n_K \bs\{p\} \to \PP^{n-1}_K$ with center $p$ is given by the inclusion $S \inkl R$. For a closed subscheme $\wt Z \tm \PP^n_K$ with $p \notin \wt Z$, denote $I_{\wt Z} \tm R$ the saturated ideal with $\wt Z = \Proj(R/I_{\wt Z})$. The homogeneous ideal of $Z = \wt\pi(\wt Z)$ is $I_Z = I_{\wt Z}\cap S$, and the outer simple linear projection $\pi = \pi\res_{\wt Z}: \wt Z \Epi Z$ is given by $S/I_{\wt Z} \inkl R/I_Z$. Let $x \in R_1\bs\idp_1$ be a linear form not vanishing in $p$. We identify $R = S[x]$ and hence can consider an element $f \in R$ as a polynomial in $x$ over $S$. We denote by $\deg_x(f)$ its degree in $x$ and by $\LC_x(f)$ its leading coefficient in $S$. For $k \in \Nat$, the {\it $k$-th partial elimination ideal of $I_{\wt Z}$ with respect to $p$} is defined by \[\sK_k^p(I_{\wt Z}) \,\defin\, \{f_0 \in S \mid \exists \wt f \in R: \deg_x(\wt f)<k \und x^kf_0+\wt f \in I_{\wt Z}\},\] a graded ideal in $S$ that does not depend on our choice of $x$ (compare \cite{ego}). The partial elimination ideals form an ascending chain of ideals \[\sK_0^p(I_{\wt Z}) = I_Z \,\tm\, \sK_1^p(I_{\wt Z}) \,\tm\, \sK_1^p(I_{\wt Z}) \,\tm\, \cdots\] Note that, as $p \notin \wt Z$, there is an integer $l \in \NN$ such that $x^l \in I_{\wt Z} + \idp$, hence $1 \in \sK^p_l(I_{\wt Z})$ and $\sK^p_l(I_{\wt Z}) = S$. 

\begin{thm}\label{PEI}
 For all $k \in \NN$, set-theoretically $\sK_{k-1}^p(I_{\wt Z})$ is the ideal of $Z_k$.
\end{thm}

For a proof of this proposition, see \cite[Theorem 3.5]{CS} or \cite[Corollary 3.5]{ego}. Now, fix a closed point $q \in Z$ and a linear form $y \in S_1 \bs \idq_1$ not vanishing in $q$. Then, $K[x,y] = R/\idq R$ is the homogeneous coordinate ring of the projective line $\langle q,p\rangle = \PP^1_K$ spanned by $p$ and $q$ in $\PP^n_K$. The homogeneous ideal \[I(\pi^{-1}(q)) = (I_{\wt Z}+\idq R)\sat/\idq R\] of $\pi^{-1}(q) = \langle q,p\rangle\cap\wt Z$ is a saturated ideal in $K[x,y]$, hence it is a principal ideal. Let $F_q \in K[x,y]$ be a homogeneous generator of $I(\pi^{-1}(q))$. For $k \in \NN$, it holds  \[q \in Z_k^\circ \,\Leftrightarrow\, \left\{\begin{array}{lclclclcl} k &=& \length_{\sO_{\PP^1_K}}(\sO_{\pi^{-1}(q)})\\ &=& e_0(K[x,y]/F_qK[x,y]) &=& \deg(F_q)\end{array}\right.\] where $e_0$ denotes the Hilbert multiplicity; see \cite[Section 3]{ego} for these equalities. Hence, if $q \in Z_k^\circ$, the polynomial $F_q$ is a binary form of degree $k$, and as such it is the product of $k$ linear forms $L_1, \ldots, L_k \in K[x,y]_1$. We identify two binary forms of the same degree if they only differ by a factor $\kappa \in K^\ast$, i.e., if they are equal up to multiplication with units. A closed point $\wt q \in \pi^{-1}(q)$ corresponds to a linear factor $L_{\wt q}$ of $F_q$, and for any partition $\lambda$ of $k$, it holds \[\pi^{-1}(q) = \sum_{i=1}^e\lambda_i\wt q_i \,\Leftrightarrow\, F_q = \prod_{i=1}^eL_{\wt q_i}^{\lambda_i}.\] 
Thus, to determine the number $e$ of closed points in the fiber $\pi^{-1}(q)$ and their multiplicities $\lambda_1, \ldots, \lambda_e$ is the same as to determine the number of distinct linear factors of $F_q$ and their multiplicities. Any binary form $F \in K[x,y]$ of degree $k \in \Nat$ can be written $F = \sum_{j=0}^ka_jx^{k-j}y^j$ with $a_0, \ldots, a_k \in K$. For a partition $\lambda$ of $k$, the {\it $\lambda$-coincident root locus $X_\lambda$} is the projective variety in $\PP^k_K$ of binary forms of degree $k$ whose linear factors are distributed according to $\lambda$, i.e., \[X_\lambda = \{(a_0:\cdots:a_k) \in \PP^k_K \mid \sum_{j=0}^ka_jx^{k-j}y^j = \prod_{i=1}^eL_i^{\lambda_i} \text{ for } L_1, \ldots, L_e \in K[x,y]_1\}.\] Note that the linear factors $L_1, \ldots, L_e$ need not be different; corresponding to the closed points in $\pi^{-1}(q)$, some of them might coincide. In this situation, it holds $\sum_{j=0}^ka_jx^{k-j}y^j \in X_\mu$ for some coarsening $\mu$ of $\lambda$. We denote by $X_\lambda^\circ$ the open subset of $X_\lambda$ consisting of all binary forms $F = \prod_{i-1}^eL_i^{\lambda_i}$ with distinct linear factors $L_1, \ldots, L_e$, i.e., \[X_\lambda^\circ = X_\lambda\bs\left(\bigcup\{X_\mu\mid \mu \text{ a coarsening of } \lambda \text{ with } \mu \neq \lambda\right).\] For an in-depth study of coincident root loci, see \cite{Chi}. 
Now, for a closed point $q \in Z$, fix $y \in S_1\bs\idq$. For any $f \in R$, we denote the restriction of $f$ to the line $\langle q,p\rangle$ by $f\res_q$; more precisely, \[f\res_q \,\,\defin\,\, f\res_{\langle q,p\rangle} \,\,=\,\, f+\idq \,\,\in\,\, K[x,y] = R/\idq R.\] For any set $G \tm R = S[x]$ of polynomials in $x$ over $S$ and all $k \in \Nat$, we denote by \[G_{x,k} \,\defin\, \{f \in G \mid \deg_x(f) \leq k\} \quad \text{and} \quad G_{x,k}^\circ \,\defin\, G_{x,k}\bs G_{x,k-1}\] the subsets of all polynomials whose degree in $x$ is less than $k+1$ and equal to $k$, respectively. Also, for $f \in R$ with $\deg_x(f) = k$, we denote by \[D_f \,\defin\, Z_k^\circ\bs V(\LC_x(f))\] the set of closed points of $Z_k^\circ$ in which the leading coefficient of $f$ with respect to $x$ does not vanish. Note that $D_f$ is an affine open subset of $Z_k^\circ$. With this notations, we are finally able to give an affine covering of $Z_k^\circ$ describing the ramification of $\pi$. 

\begin{thm}\label{covering}
 Let $k \in \NN$. The sets $D_f$ for $f \in (I_{\wt Z})_{x,k}^\circ$ form an affine covering of $Z_k^\circ$. Moreover, for all $f \in (I_{\wt Z})_{x,k}^\circ$, all $q \in D_f$ and all partitions $\lambda$ of $k$, it holds \[q \in Z_\lambda^\circ \,\Leftrightarrow\, f\res_q \in X_\lambda^\circ.\]
\end{thm}

\begin{proof}
 Let $q \in Z_k^\circ$ be a closed point with homogeneous ideal $\idq \in \Proj(S)$. By Proposition \ref{PEI}, it holds $\sK_k^p(I_{\wt Z}) \ntm \idq$ and $\sK_l^p(I_{\wt Z}) \tm \idq$ for all $l < k$. Hence, there is an element $f \in (I_{\wt Z})_{x,k}^\circ$ with $\LC_x(f) \notin \idq$. Thus $q \in D_f$, and \[0 \neq \LC_x(f\res_q) = \LC_x(f)\res_q \in S/\idq.\] It follows $\deg_x(f\res_q) = k$. As $I(\pi^{-1}(q)) \tm R/\idq R$ is generated by an element of degree $k$, this implies $I(\pi^{-1}(q)) = f\res_q(R/\idq R)$, and we get our claim by the above considerations.
\end{proof}

The above Proposition can be reformulated in a stronger way. More precisely, we can determine finitely many elements of $(I_{\wt Z})_{x,k}^\circ$ which already describe the ramification of $\pi$ by use of Gr\"obner bases.

\begin{thm}\label{Gro}
 Assume $p = (1:0:\cdots:0)$, and let $G$ be a Gr\"obner basis of $I_{\wt Z}$ with respect to an elimination ordering $\sigma$. Then, for all $k \in \Nat$, the leading coefficients $\LC_{x_0}(g)$ for $g \in G_{x_0,k}$ form a Gr\"obner basis of $\sK_k^p(I_{\wt Z})$ with respect to the ordering induced by $\sigma$ on $S$.
\end{thm}

For a proof, see \cite[Proposition 3.4]{CS} and observe that their definition of partial elimination ideals coincides with ours for $p = (1:0:\cdots:0)$ and $x = x_0 \in R_1\bs\idp$.

\begin{cor}\label{GBcovering}
 Assume $p = (1:0:\cdots:0)$, and let $G$ be a finite Gr\"obner basis of $I_{\wt Z}$ with respect to an elimination ordering. The sets $D_g$ for $g \in G_{x_0,k}^\circ$ form a finite affine covering of $Z_k^\circ$. Moreover, for all $g \in G_{x_0,k}^\circ$, all $q \in D_g$, and all partitions $\lambda$ of $k$, it holds \[q \in Z_\lambda^\circ \,\Leftrightarrow\, g\res_q \in X_\lambda^\circ.\]
\end{cor}

\begin{proof} This is clear by the Propositions \ref{covering} and \ref{Gro}.
\end{proof}

The last Corollary also applies mutatis mutandi if $p \neq (1:0:\cdots:0)$. Indeed, the number of linear factors of a binary form and their distribution do not change under a coordinate transformation (compare \cite{Chi}). Also, partial elimination ideals commute with coordinate transformations (see \cite[Lemma 2.6]{ego}). Hence, for an arbitrary projection center $p \in \PP^n_K$, we can choose a graded automorphism $\psi: R \overset\cong\to R$ with $\psi(\idp) = \langle x_1, \ldots, x_n\rangle$ and compute a Gr\"obner basis $G$ of $\psi(I_{\wt Z})$ with respect to an elimination ordering. The elements of $\psi^{-1}(G_{x_0,k}^\circ)$ then form a finite afffine covering describing the ramification of $\pi$ as in Corollary \ref{GBcovering}.

\section{Equations for $Z_\lambda$}

Fix $k \in \NN$. We denote by $K[z_0, \ldots, z_k]$ the homogeneous coordinate ring of $\PP^k_K$ and consider the natural inclusion $\iota: K[z_0, \ldots, z_k] \inkl S[z_0, \ldots, z_k]$. If $f \in R_{x,k}$, i.e., if $\deg_{x}(f) \leq k$, then we can write $f = \sum_{i=0}^kf_ix^{k-i}$ with $f_0, \ldots, f_k \in S$. Accordingly, for $F \in K[z_0, \ldots, z_k]$, we set \begin{equation}\label{F(f)}F(f) \,\defin\, \iota(F)(f_0, \ldots, f_k) \,\in\, S.\end{equation}
For a partition $\lambda$ of $k$, we can compute the ideal $I(X_\lambda)$ of the coincident root locus $X_\lambda$ in $K[z_0, \ldots, z_k]$ (compare \cite[Section 3.1]{Chi} or \cite[Remark 1.5]{ego2}).

\begin{them}\label{Satz}
 Let $\lambda$ be a partition of $k \in \NN$. Then, $Z_\lambda$ is set-theortically defined by the ideal generated by all elements of the form $F(f)$ with $F \in I(X_\lambda)$ and $f \in (I_{\wt Z})_{x,k}^\circ$ together with $\sK_{k-1}^p(I_{\wt Z})$.
\end{them}

\begin{proof}
 Observe that $Z_k$ is defined by $\sK_{k-1}^p(I_{\wt Z})$ and that $Z_\lambda \tm Z_k$. We show that for a closed point $q \in Z_k$, the polynomials $F(f) \in S$ of our claim are contained in the ideal $\idq$ of $q$ if and only if $q \in Z_\lambda$. First, we note that according to \cite[Remark 3.2]{ego2}, the ideal $I(X_\lambda) \tm K[z_0, \ldots, z_k]$ is graded with respect to the grading induced by the weight $\omega = (0,1, \ldots, k)$, i.e., $\omega(z_j) = j$ for all $j \in \{0, \ldots, k\}$. Hence, we only need show our claim for polynomials $F \in I(X_\lambda)$ which are homogeneous with respect to the grading induced by $\omega$.\\
 Assume $q \in Z_\lambda$. If $q \in Z_{k+1}$, then $I(\pi^{-1}(q)) \tm R/\idq R$ is generated in degree $> k$. Assume that there is an element $f \in (I_{\wt Z})_{x,k}^\circ$ not contained in $\idq R$. Then $0 \neq f+\idq \in I(\pi^{-1}(q)) \tm R/\idq R$ and $\deg_x(f+\idq) = k$. But $I(\pi^{-1}(q))$ is zero in degree $k$ -- a contradiction. Hence, for all $f = \sum_{i=0}^kf_0x^{k-i} \in (I_{\wt Z})_{x,k}^\circ$ it holds $f \in \idq R$. As $x \notin \idq$, it follows $f_i \in \idq$ for all $i \in \{0, \ldots, k\}$, and thus $F(f) \in \idq$ for all $F \in K[z_0, \ldots, z_k]$.\\
 So, let $q \in Z_\lambda\cap Z_k^\circ$. For $f = \sum_{i=0}^kf_ix^{k-i} \in (I_{\wt Z})_{x,k}^\circ$, it either holds $f \in \idq R$ and hence $F(f) \in \idq$ for all $F \in K[z_0, \ldots, z_k]$ as above, or $f \notin \idq R$. In the latter case, $f\res_q \in X_\lambda$ by Proposition \ref{covering}. Let $y \in S_1\backslash\idq$, and denote by $t \defin \deg(f)$ the total degree of $f \in R$. Then, \[(\ol{f}_0, \ldots, \ol f_k) \,\defin\, \left(\frac{f_0+\idq}{y^{t-k}}, \frac{f_1+\idq}{y^{t-k+1}}, \ldots, \frac{f_k+\idq}{y^{t}}\right) \,\in\, K^{k+1}\] are coefficients of the binary form $f\res_q$. For a polynomial $F\in I(X_\lambda)$ which is homogeneous of degree $m$ with respect to the standard grading and homogeneous of degree $s$ with respect to the grading induced by $\omega$, we compute \[\begin{array}{rclcl} F(f_0+\idq, \ldots, f_k+\idq) &=& F(y^{t-k}\ol f_0, \ldots, y^t\ol f_k)\\ &=& y^{(t-k)m+s}F(\ol{f}_0, \ldots, \ol f_k) &=& 0\end{array}\] since $f\res_q \in X_\lambda$. Hence $F(f) \in \idq$. This shows the implication ``$\Rightarrow$''.\\
 Now, assume that $F(f) \in \idq$ for all $F\in I(X_\lambda)$ and all $f \in (I_{\wt Z})_{x,k}^\circ$. If $f \in \idq R$ for all $f \in (I_{\wt Z})_{x,k}^\circ$, then $\sK_k^p(I_{\wt Z}) \tm \idq$, and hence $q \in Z_{k+1}\tm Z_\lambda$. On the other hand, if there is an element $f = \sum_{i=0}^kf_ix^{k-i} \in (I_{\wt Z})_{x,k}^\circ$ with $f \notin \idq R$, then $F(f) \in \idq$ implies $F(f_0+\idq, \ldots, f_k+\idq) = 0$. If $F$ is homogeneous with respect to the standard grading as well as the grading induced by $\omega$, with the same notations and arguments as before we get $F(\ol f_0, \ldots, \ol f_k) = 0$, hence $f\res_q \in X_\lambda$. Another application of Proposition \ref{covering} finishes the proof.
\end{proof}

Again, we can use Gr\"obner bases to give a version of the above Theorem more suited to explicite computations. 

\begin{cor}\label{last}
Let $\lambda = (\lambda_1, \ldots, \lambda_e)$ be a partition of $k \in \Nat$. Assume in addition that $p = (1:0:\cdots:0)$, and let $G$ be a Gr\"obner Basis of $I_{\wt Z}$ with respect to an elimination ordering on $R$. Let $F_1, \ldots, F_s$ be generators of $I(X_\lambda) \in K[z_0, \ldots, z_k]$. Then, set-theoretically, \[I(Z_\lambda) = \left\langle\{\LC_{x_0}(g) \mid g \in G_{x,k-1}\}\cup \{F_i(g)\mid i \in \{1, \ldots, s\},g \in G_{x,k}^\circ\}\right\rangle.\]
\end{cor}

\begin{proof} 
 By Propositions \ref{PEI} and \ref{Gro}, the leading coefficients $\LC_{x_0}(g)$ in $x_0$ of the polynomials $g \in G_{x,k-1}$ define the scheme $Z_k$ set-theoretically. According to Corollary \ref{GBcovering}, a closed point $q \in Z_k$ belongs to $Z_\lambda$ if and only if for all $g \in G_{x,k}^\circ$ with $g \notin \idq R$ it holds $g\res_q \in X_\lambda$. This is equivalent to $F_i(g) \in \idq$ for all $i \in \{1, \ldots, s\}$. 
\end{proof}

As before, we can also apply the above Corollary if $p \neq (1:0:\cdots:0)$ by use of a graded automorphism $\psi: R \stackrel{\cong}{\to} R$ with $\psi(\idq) = \langle x_1, \ldots, x_n\rangle$. For a Gr\"obner basis $G$ of $\psi(I_{\wt Z})$ with respect to an elimination ordering, the ideal of $Z_\lambda$ is set-theoretically generated by all polynomials $\psi^{-1}(\LC_{x_0}(g))$ for $g \in G_{x,k-1}$ and $\psi^{-1}(F_i(g))$ for $g \in G_{x,k}^\circ$ where $F_1, \ldots, F_s$ are generators of $I(X_\lambda)$.

\medskip

{\bf Acknowledgments:} I thank my Ph.D. advisor Markus Brodmann for careful proofreading and reminding me of my intention to keep this paper short.


\begin{thebibliography}{PEI}
\bibitem[BrP]{BrP} M. Brodmann and E. Park, {\it On varieties of almost minimal degree I: Secant loci of rational normal scrolls}. J. Pure Appl. Algebra 214 (2010), no. 11, 2033-2043. 
\bibitem{Chi} J. Chipalkatti, \textit{On equations defining Coincident Root loci}. J. Algebra 267 (2003), no. 1, 246-271.
\bibitem{CS} A. Conca and J. Sidman, {\it Generic initial ideals of points and curves}. in J. Symbolic Comput. 40 (2005), no. 3, pp. 1023-1038.

\bibitem{G} M. Green, {\it Generic initial ideals}. In J. Elias et.al., Six Lectures on Commutative Algebra, Progress in Mathematics 166, Birkh{\"a}user, Basel, 1998, pp. 119-186.
\bibitem{ego} S. Kurmann, {\it Partial elimination ideals and secant cones}. J. Algebra 327 (2011), 489-505. 
\bibitem{ego2} S. Kurmann, {\it Some comments on equations defining Coincident Root Loci}. Preprint, arXiv:1108.4532v1 [math.AG]
\bibitem{Diss} S. Kurmann, {\it Equations for the Ramification Loci of Outer Simple Linear Projections}. Doctoral thesis at the University of Z\"urich, 2011.
\end{thebibliography}
\end{document}